\documentclass[11pt]{amsart}

\usepackage[margin=1.15in]{geometry}
\usepackage{amsmath,amssymb,amsthm}
\usepackage{hyperref}
\usepackage{tikz}

\hypersetup{
  colorlinks=true,
  linkcolor=purple,
  citecolor=teal,
}

\newtheorem{theorem}{Theorem}[section]
\newtheorem{corollary}[theorem]{Corollary}
\newtheorem{proposition}[theorem]{Proposition}

\theoremstyle{definition}
\newtheorem{definition}[theorem]{Definition}
\theoremstyle{remark}

\newcommand{\CFK}{CFK^{\infty}}
\newcommand{\CFhat}{\widehat{CF}}
\newcommand{\HFKhat}{\widehat{HFK}}
\newcommand{\HFhat}{\widehat{HF}}
\newcommand{\cF}{\mathcal F}

\def\DrawContactLevel{1}
\usetikzlibrary{arrows.meta,patterns,decorations.pathreplacing}
\definecolor{ink}{HTML}{262626}
\definecolor{rhocolor}{HTML}{35688D}
\definecolor{rhofill}{HTML}{E8F0F5}
\definecolor{levelcolor}{HTML}{24745B}
\newif\ifcontactlevel
\ifdefined\DrawContactLevel\contactleveltrue\else\contactlevelfalse\fi

\title[ The $b$-invariant]{ On the Upsilon invariant and the Baldwin--Vela-Vick $b$-invariant of fibered knots}

\author{Linh Truong}
\address{Department of Mathematics, University of Michigan, Ann Arbor, MI 48109, U.S.A.}
\email{tlinh@umich.edu}

\begin{document}

\begin{abstract}

For a null-homologous fibered knot $K$ in a rational homology sphere, Baldwin and Vela-Vick define a refinement $b(K)$ of the Ozsv\'ath-Szab\'o contact invariant. We prove that if the Ozsv\'ath-Stipsicz-Szab\'o  concordance invariant Upsilon has  slope $\Upsilon'_K(t)=-g(K)$ for some $t\in[0,\frac{2}{n+1})$, where $g(K)$ is the genus of $K$, then $b(K)>n$. This extends a result of He, Hubbard, and the author on the Upsilon invariant and right-veering monodromies. We also prove that the knot Floer homology of a fibered knot in the $g(K)-b(K)$ Alexander grading is nonvanishing. As an application, we give many families of fibered cable knots whose $b(K)$ invariant grows arbitrarily large. Finally, we prove that if a monodromy of a fibered knot is neither right-veering nor left-veering, then the third-to-top term of knot Floer homology is nonvanishing. 
\end{abstract}

\maketitle

\section{Introduction}

Given a fibered knot $K$ in a closed 3-manifold  $Y$, its associated open book decomposition determines a contact structure under the Giroux correspondence. This contact structure $\xi_K$ has an associated  contact class in the Heegaard Floer homology group of $-Y$
\[c(\xi_K) = [c] \in \HFhat(-Y),\]
a contact invariant of the contact manifold $(Y, \xi_K)$ defined by Ozsv\'ath and Szab\'o \cite{OS-contact}. A refinement of the contact invariant $b(K)$ was introduced by Baldwin and Vela-Vick \cite{BVV}, obtained from the filtration of the Heegaard Floer complex of $-Y$ induced from the fibered knot $K$.  The invariant $b(K)$ takes values in $\mathbb{Z}_+\cup \{\infty \}$. If the contact invariant $c(\xi_K)$ vanishes, then $b(K)$ measures the minimum  filtration level at which the contact class vanishes. If the contact invariant $c(\xi_K) \neq 0$, then $b(K) = \infty$. Using a Heegaard diagram adapted to the fibered knot $K \subset Y$, Baldwin and Vela-Vick proved that $b(K) > 1 $ implies the monodromy associated to $K$ is right-veering. Recently, Baldwin, Ni, and Sivek proved the converse \cite{BNS}, implying the knot Floer homology detects right-veeringness, by leveraging relationships between knot Floer homology and symplectic Floer homology.  

The Upsilon invariant defined by Ozsv\'ath, Stipsicz, and Szab\'o \cite{OSS} for knots in $S^3$  is a piecewise linear smooth concordance invariant obtained from the filtered knot Floer complex $\CFK(S^3,K)$. Its generalization to
null-homologous knots in rational homology spheres was introduced by Alfieri, Celoria, and Stipsicz \cite{ACS}. 
He, Hubbard, and the author \cite{HHT} proved that for a fibered knot, if the Upsilon invariant has a slope $-g(K)$ at some $t \in [0,1)$, where $g(K)$ is the genus of the knot, then the monodromy of the associated open book is right-veering. Since the right-veering  condition is equivalent to the inequality $b(K)>1$, it is natural to ask whether we can extract more information on the $b$-invariant depending on Upsilon. Indeed, by extending the techniques used in \cite{HHT}, we will show that the location $t$ at which Upsilon attains slope $-g(K)$ will provide a lower bound on $b(K)$.

\begin{theorem}\label{thm:main}
Let $K\subset Y$ be a null-homologous fibered knot of genus $g(K)>0$ in a rational homology sphere $Y$, and let $n \geq 1$ be an integer.  
If 
$\Upsilon'_{Y,K, \mathfrak{s}}(t)=-g(K)$ for some non-singular point $t\in[0,\frac{2}{n+1})$ and some spin$^c$ structure $\mathfrak{s} $ of $Y$, then $b(K)>n$. 

\end{theorem}

Note that for $n=1$ this theorem recovers the main result of \cite{HHT}. Since the right-hand slope of Upsilon at $t=0$ recovers (the negative of) the $\tau$ invariant,  we additionally recover Hedden's result \cite{Hedden} that  fibered knots in $S^3$ with $\tau(K) = g(K)$ induce the (unique tight) contact structure on $S^3$ with nonvanishing Heegaard Floer contact invariant.

Since Upsilon is a concordance invariant, we obtain a concordance obstruction among fibered knots.

\begin{corollary}\label{cor:concordance}
Let $K,J\subset S^3$ be fibered knots with equal genus $g(K)=g(J)=g$,
and suppose that $K$ and $J$ are smoothly concordant. 
If
\[
\Upsilon'_K(t_0)=-g,
\]
for some $t_0 \in [0,\frac{2}{n+1})$
then
\[ b(J)\geq n. \]
\end{corollary}

\begin{proof}
Since $\Upsilon$ is a smooth concordance invariant,
$\Upsilon_J=\Upsilon_K$.
Thus
\[
\Upsilon'_J(t_0)
=
\Upsilon'_K(t_0)
=
-g
=
-g(J).
\]
The result follows from Theorem~\ref{thm:main} if $g>0$. If $g=0$, then $K$ and $J$ are unknots and $b(K)=b(J) = \infty$.
\end{proof}

This theorem and its predecessor \cite{HHT} suggest that knots that achieve minimal Upsilon slope (i.e. negative genus) are an interesting class of knots. See \cite[Corollary 1.7]{HHT}.
We also have immediate consequences for their concordance genus. 
Let $K$ be a knot with $\Upsilon_K'(t) = -g(K)$ for some $t \in (0,1)$. Consider the concordance genus $g_c(K)$, which is the minimal
Seifert genus of any knot concordant to $K$, and the stable concordance genus  $ \underrightarrow{g_c}(K) = \lim_{n \to \infty} \frac{g_c(nK)}{n}$. Then 
\[g(K) = g_c(K
) = \underrightarrow{g_c}(K).\] This follows from \cite[Theorem 1.13]{OSS}. That is, if $J$ is concordant to $K$, then $J$ has the same Upsilon invariant as $K$. However, the slopes of Upsilon of $J$ are bounded below by $-g(J)$, and thus $g(J) \geq g(K)$. Thus, $K$ is genus-minimizing in its concordance class, or in other words, $g_c(K) = g(K)$. Moreover, because Upsilon is additive under connected sums, for this class of knots we have that the concordance genus of a connected sum of $K$ is additive. As a consequence, 
$ \underrightarrow{g_c} (K) =g(K)$. Concordance genus is studied in \cite{Livingston-concordancegenus, Hom-concordancegenus, OSS}. Stable concordance genus is studied in \cite{Kearney}.

In another direction, Baldwin-Vela-Vick show that the knot Floer homology of a fibered knot is nontrivial in its next-to-top Alexander grading. We prove an analogous non-vanishing statement: the term of knot Floer homology of a fibered knot in Alexander grading $g(K)-b(K)$ is nonzero. 

\begin{proposition} \label{prop:nonzero}
    Suppose $K \subset Y $ is a fibered knot of genus $g > 0$ with fiber surface $\Sigma$ and $b(K) = n < \infty$. Then the summand $\HFKhat(Y, K, [\Sigma], g-n)$ is nonzero. 
\end{proposition}
We remark that an alternative proof of Proposition~\ref{prop:nonzero} follows from \cite[Remark 1.4]{BNS}. 
For $n = 1$, this statement is proven in \cite[Theorem 1.1]{BVV}; our proof is analogous to their arguments.  In the next proposition,  we give many families of cable knots which have arbitrarily large $b$-invariant. Instead of computing Upsilon, the proof rather relies on our nonvanishing statement  Proposition~\ref{prop:nonzero}, combined with work of Hedden \cite{Hedden-cabling} on the knot Floer homology of cable knots.

\begin{proposition} \label{prop:cableexamples}
Let $n \geq 0$ and $p \geq 2$ be  integers. 
Let $K$ be a nontrivial fibered knot in $S^3$ with $\tau(K) < g(K)$. There exists an integer $N_p$ such that for $n > N_p$,  
the cable knots $K_{p, pn+1}$  have  $b$-invariant bounds
\[ p \leq b(K_{p, pn+1}) < \infty\]
\end{proposition}
\begin{proof}
    Since $\tau(K) < g(K)$, the contact structure associated to the fibered knot $K$, denoted $\xi_K$, has vanishing Heegaard Floer contact invariant $c(\xi_K) = 0$ \cite{Hedden}. Since the contact structure associated to the positive $(p,pn+1)$ cable of the knot (which is also fibered), denoted $\xi_{p,pn+1}$, is the same \cite{Hedden-contact} as $\xi_K$, its contact invariant $c(\xi_{p,pn+1}) = 0$ must also vanish. 
    Therefore, $b(K_{p, pn+1})$ is finite. 
    
    For  cabling coefficients $p > 1, q >0$, the cable $K_{p,q}$ is fibered if $K$ is fibered, and its monodromy has fractional Dehn twist coefficient
    \[ c(\phi_{K_{p,q}})=\frac{1}{pq}>0 \] by \cite[Proposition~4.2]{KazezRoberts}. Hence its monodromy is right-veering. Right-veeringness of positive    cables implies that $b(K_{p, pn+1}) > 1$.
    
    We use  \cite[Lemma 3.6]{Hedden-cabling} to conclude that $b(K_{p, pn+1}) > p-1$. Let $N_p$ be the integer determined in \cite[Lemma 3.6]{Hedden-cabling}. Then the terms $\HFKhat(S^3,K_{p, pn+1}, i-g)$ of knot Floer homology  $i=2, \dots, p-1$ are zero. Thus, by Proposition~\ref{prop:nonzero} the $b$-invariant cannot equal any of $2, \dots, p-1$. Therefore, $b(K_
    {p, pn+1})>p-1$.
\end{proof}

The non-vanishing of the next-to-top term of knot Floer homology for fibered knots has found uses in knot detection results. We expect that Proposition~\ref{prop:nonzero} and Theorem~\ref{thm:rankthirdtotop} below could similarly be useful in knot detection applications for fibered knots. 

By extending techniques of Baldwin-Vela-Vick \cite{BVV}, Ni \cite{Ni} shows that the rank of the second-to-top term of knot Floer homology must be greater than 1 if the monodromy is neither right-veering nor left-veering. In a closely related direction to \cite[Theorem A.1]{Ni}, we also show that the third-to-top term of knot Floer homology of a fibered knot is nonzero when the monodromy is neither right-veering nor left-veering.   

\begin{theorem} \label{thm:rankthirdtotop}
Let $Y$ be a closed, oriented 3-manifold. 
    If $K$ is a fibered knot in $Y$ with fiber surface $\Sigma$ whose monodromy is not right-veering and not left-veering, then 
    \[\text{rank}(\HFKhat(Y, K, [\Sigma],  2-g)) \neq 0.\]
\end{theorem}
 Theorem~\ref{thm:thirdtotop} states a  stronger conclusion than Theorem~\ref{thm:rankthirdtotop}. 

\subsection{Acknowledgements} The author was partially supported with an NSF CAREER grant DMS-2237131.

\subsection{AI Disclosure} ChatGPT Sol assisted in computing the Upsilon function for the cable knots in  Corollary~\ref{example:upsiloncables}.  ChatGPT Astra was used to create a figure.

\section{Background}

\subsection{Conventions} We work over $\mathbb{F}=\mathbb Z/2\mathbb Z$. We will assume the reader is familiar with knot Floer homology \cite
{OS-knots}.  
 Let $K$ be a null-homologous knot in a rational homology 3-sphere $Y$. We will sometimes work with the mirror image $-(Y,K, \mathfrak{s}) = (-Y, -K, \overline{\mathfrak{s}}) $ of $(Y,K, \mathfrak{s})$. Here $-Y$ denotes $Y$ with the reverse orientation and $-K$ denotes the knot $K$ with the reverse orientation.

A generator of the knot Floer complex $\CFK(Y,K
)$, a module over $\mathbb{F}[U, U^{-1}]$, will be written as 
\[
        [x,i,j],
\]
where $i$ is the algebraic filtration level and $j$ is the Alexander filtration level. A bold generator $\mathbf{x}$ refers to its algebraic filtration level $i=0$ representative $\mathbf{x}= [x, 0, A(x)]$.
The differential in $\CFK(Y,K)$  decreases (non-strictly) both the $i$- and $j$-filtrations.  The action of $U$ is given by 
\[
        U[x,i,j]=[x,i-1,j-1].
\]

\subsection{Upsilon invariant}
We recall the definition of the Upsilon function $\Upsilon_K(t)$ of a knot \cite{OSS}. 
For $t\in[0,2]$, define the $t$-filtration on the knot Floer complex as follows.  
For a chain $\theta=\sum_k [x_k,i_k,j_k]$, define
\[
        F_t(\theta)=\max_k f_t([x_k,i_k,j_k])
\]
where 
\[
        f_t([x,i,j])=\bigl(1-\tfrac t2\bigr)i+\tfrac t2 j.
\]
For a fixed $t$, the function $F_t$ defines a real-valued function on the knot Floer complex, and descends to a filtration on the knot Floer complex (cf. \cite{Livingston}). For each real number $s$, the $t$-filtration level $s$ of $\CFK(Y,K, \mathfrak{s})$ is the subcomplex spanned by elements $\theta \in \CFK(Y,K, \mathfrak{s}) $ where $F_t(\theta) \leq s$. 
Following Livingston's reformulation of $\Upsilon_K(t)$ for knots in $S^3$, which was generalized to the setting of null-homologous knots in rational homology 3-spheres in \cite{ACS}, we define Upsilon as follows.

\begin{definition}
Let $ \CFK(Y,K, \mathfrak{s})$ be the knot Floer chain complex for a null-homologous knot $K$ in a rational homology 3-sphere $Y$ and a spin$^c$ structure $\mathfrak{s}$ for $Y$. Let 
\begin{align*}
        \nu_t(\CFK(Y,K, \mathfrak{s}))=\min\{F_t(\theta)\mid \theta \in \CFK(Y,K, \mathfrak{s})
        \text{ is a homogeneous cycle } \\ \text{of Maslov grading } d(Y, \mathfrak{s}),
        \\
        \ [\theta]\neq 0\in H_*(\CFK(Y,K, \mathfrak{s}))\}
\end{align*}
and define the Upsilon invariant as
\[
        \Upsilon_{Y,K, \mathfrak{s}}(t)=-2\nu_t(\CFK(Y,K, \mathfrak{s})).
\]
\end{definition}
The Upsilon invariant is a knot concordance invariant \cite{OSS} in the classical sense for knots in $S^3$. Upsilon also gives a spin$^c$ rational concordance invariant for null-homologous knots in rational homology 3-spheres $Y$ equipped with a spin$^c$ structure $\mathfrak{s}$ \cite{ACS}.

\subsection{The contact invariant}  \label{section:contact}
We review the Honda-Kazez-Mati\'c description of the contact invariant \cite{HondaKazezMatic}.
Suppose $K \subset Y$ is a fibered knot of genus $g$ and let $(\Sigma, \phi)$ be an open book corresponding to $K$. Choose a basis $\{a_1, \dots, a_{2g}\}$ for $\Sigma$, that is, a set of properly embedded, disjoint arcs in $\Sigma$ such that the complement $\Sigma \setminus \cup_i \alpha_i$ is a disk. Let $b_i$ be an isotopic copy of $a_i$, by taking a small translate, where the endpoints of $b_i$ are obtained by pushing the endpoints of $a_i$ in the direction of the orientation of $\partial \Sigma$, and taken so that there is a single transverse intersection point 
$a_i \cap b_i = \{c_i\}$. 
We will construct a Heegaard diagram \[(S, \boldsymbol{\alpha} = \{\alpha_1, \dots, \alpha_{2g}\}, \boldsymbol{\beta}=\{\beta_1, \dots, \beta_{2g}\}, w)\] for $Y$ adapted to the open book, which consists of the data:
\begin{itemize}
    \item $S = \Sigma \cup -\Sigma$  is the ``double" of the fiber surface, that is, the union of two copies of $\Sigma$ glued along their boundaries,
    \item $\alpha_i = a_i \cup a_i \subset S$ is the ``double" of the arc $a_i \subset \Sigma$, that is, the union of two copies of $a_i$ glued along their endpoints,
    \item $\beta_i = b_i \cup \phi(b_i) \subset S$,
    \item $w \in \Sigma \setminus \cup_{\alpha_i \in \boldsymbol{\alpha}} \alpha_i \setminus \cup_{\beta_i \in \boldsymbol{\beta}} \beta_i \subset S $ is a basepoint in the complement of the curves in $\boldsymbol \alpha$ and $\boldsymbol \beta$ and the thin regions bounded by the isotopies between $a_i$ and $b_i$.
\end{itemize}
By reversing the roles of $\boldsymbol{\beta}$ and $ \boldsymbol{\alpha} $, we have a Heegaard diagram $(S, \boldsymbol{\beta}, \boldsymbol{\alpha}, w)$ for $-Y$.
Honda, Kazez, and Mati\'c show that the generator
\[\boldsymbol{c} := \{c_1, \dots, c_{2g}\}\]
is a cycle whose homology class $[\boldsymbol{c}]$ in $\HFhat(-Y)$ coincides with the contact class, an invariant for contact structures originally defined by Ozsv\'ath and Szab\'o \cite{OS-contact}. 

We can also construct a doubly pointed Heegaard diagram for $K \subset Y$ as follows.  Starting in the region next to $b_1$ in the direction of the orientation of the boundary  $\partial \Sigma$, we perform finger moves on the $\beta_i$ curves in a small neighborhood of the boundary $\partial \Sigma$ pushing in the direction of the orientation on $\partial \Sigma$, until we return to the start. We place an additional basepoint $z$ in the region next to $a_1$ and the pushed $\beta_i $ curves.  Then $\partial \Sigma$ is the union of an arc from $z$ to $w$ in the complement of the $\alpha$ curves and an arc from $w$ to $z$ in the complement of the $\beta$ curves. 

We will consider $\mathcal{H} = (S, \boldsymbol{\beta}, \boldsymbol{\alpha}, z, w)$, the Heegaard diagram for $-K \subset -Y$ obtained by reversing the roles of $\boldsymbol{\beta}$ and $ \boldsymbol{\alpha} $ and the roles of $z$ and $w$.  The contact generator $\mathbf c =\{ c_1, \dots, c_{2g}\}$ is also a cycle in the complex $ \widehat{CFK}(\mathcal{H})$. Moreover, its homology class is nonzero in the bottom-most filtration level $0 \neq [\mathbf c] \in H_*(\mathcal{F}_{-g}) \cong \mathbb{F}$ by \cite[Theorem 2.3]{BVV}. 

\subsection{The $b$-invariant}
We recall the definition of the Baldwin--Vela-Vick invariant $b(K)$ of a fibered knot $K$ in a closed, oriented 3-manifold $Y$ \cite{BVV}. Let
\[
         \cF_{-g}\subset \cF_{1-g}\subset\cdots\subset \cF_g=\CFhat(-Y)
\]
denote the Alexander filtration on $\CFhat(-Y)$ induced by the fibered knot.  The bottom filtration level is generated in homology by the contact class $[\mathbf{c}]$:
\[
        H_*(\cF_{-g})=\mathbb{F}\langle [\mathbf{c}] \rangle
\]
where $[\mathbf{c}]\in \HFhat(-Y)$
represents the Ozsvath--Szabo contact invariant of the contact structure supported by the open book of $K$.

Let 
\begin{equation}\label{eq:b-definition}
        b(K)=
        \begin{cases}
        \infty, & [\mathbf{c}]\neq 0\in \HFhat(-Y),\\[2mm]
        g+\min\{k\mid [\mathbf{c}]=0\text{ in }H_*(\cF_k)\}, & [\mathbf{c}]=0\in \HFhat(-Y).
        \end{cases}
\end{equation}
Baldwin-Ni-Sivek \cite{BNS} show that $b(K)$ detects right-veeringness of the monodromy of $K$.

\section{Bounding $b(K)$ with Upsilon}
In the proof of Theorem~\ref{thm:main}, we follow the same strategy as \cite{HHT}. The main idea is as follows. If the slope of Upsilon is $-g(K)$
 of a fibered knot $K$, then the contact generator is contributing to Upsilon. If the $b$-invariant is finite, there is a vertical arrow of length $b(K)$ in the differential of the knot Floer complex hitting the contact generator. However,  the contact generator cannot be a boundary in the full knot Floer complex $\CFK(Y,K)$ since it contributes to Upsilon. Thus, there must be a non-vertical arrow (pointing to the left (and possibly downwards)) to some other generator(s) (which have Alexander grading greater than $-g(K)$). Then at every $t$ sufficiently close to $t=0$, those other generator(s)  contribute to Upsilon instead of the contact generator.

\subsection{Proof of Theorem \ref{thm:main}}

First, we prove Proposition~\ref{prop:contactgenerator} which will be useful in the proof of Theorem~\ref{thm:main}. The proposition is a refinement of \cite[Lemma 2.1]{Hom-cables}, which shows, through a careful change of basis, that the knot Floer complex is homotopy equivalent to a reduced complex. A reduced complex means that the differential has no component preserving both algebraic and Alexander filtration levels.  That is, for every chain, the differential of that chain strictly drops either the Alexander or algebraic filtration or both. We will keep track of the contact generator throughout the procedure of the change of basis.

\begin{proposition}[cf. \cite{HHT}]\label{prop:contactgenerator}
Suppose $K \subset Y$ is a fibered knot of genus $g$ and let $(\Sigma, \phi)$ be an open book corresponding to $K$. Let $\mathcal{H} = (S, \boldsymbol{\beta}, \boldsymbol{\alpha}, z, w)$ be the Heegaard diagram for $-K \subset -Y$ described in Section~\ref{section:contact}.  There is a reduced doubly filtered complex $C$ and a filtered
chain homotopy equivalence
\[
CFK^-(\mathcal H)\simeq C
\]
which fixes the contact generator $\mathbf c$, such that
$\mathbf c$ is the unique basis element of $C$ with 
Alexander grading $-g$.
\end{proposition}

\begin{proof}
Let
\[ C^{\mathrm{vert}} = CFK^-(\mathcal H)\{i=0\} \cong
CFK^-(\mathcal H)/UCFK^-(\mathcal H) \]
 which comes equipped with differential $\partial^\text{vert}$, which we will denote by $\partial$ for the remainder of the proof. Let $\{x_i\}$ be the filtered basis (over $\mathbb{F}$) of generators $\mathbb{T}_\alpha \cap \mathbb{T}_\beta$ for $C^\text{vert} $. Note that one of the generators $\{x_i\}$ is equal to the contact generator $\mathbf{c}$. For a nonzero chain $\theta$, let $A(\theta)$ denote its Alexander
filtration level, namely the maximum Alexander grading among its
nonzero summands. 

We will perform the first step of the procedure in the proof of \cite[Lemma 2.1]{Hom-cables}, which will simplify the set $B_0$ with respect to the basis $\{x_i\}$, while keeping track of what happens to the  contact generator $\mathbf{c}$.  The set $B_0$ is defined to be the set of generators $x_i$ which have nonzero differential $\partial x_i \neq 0$ and satisfy the degree restriction $A(\partial x_i) = A(x_i )$.  Since the contact generator $\mathbf{c}$ is a cycle, $\mathbf{c} \notin B_0$. Moreover, observe that for any generator  $x \in C^\text{vert}$ with Alexander grading $A(x) = -g$ and $\partial x \neq 0$, then $x \in B_0$ since $0 = \mathcal{F}_{-g-1} \subset \mathcal{F}_{-g}$.

If $\partial x = \sum c_j x_j, $ then we define $\partial_j x = c_j$. 
For each $x_j \in B_0$, we can find a generator $x_k \neq \mathbf{c}$ such that $\partial_k x_j = 1$ and $A(x_k) = A(x_j)$. Note that we can choose $x_k \neq \mathbf{c} $, because otherwise, $\partial x_j = \mathbf{c}$, and since $x_j \in B_0$ we have that  $A(x_j) = A(\mathbf{c}) = -g$, but then $\mathbf{c}$ would vanish in $H_*(\mathcal{F}_{-g})$, contradicting \cite[Theorem 2.3]{BVV}. Next, as in  \cite[Proof of Lemma 2.1]{Hom-cables}, given a $x_j \in B_0$ we perform the filtered change of basis

\begin{align*}
    x_j' &= x_j
    \\
    x_k' &= \partial x_j
    \\
    x_l ' &= x_l + (\partial_k x_l) x_j, \ \ \ l \neq j, k
\end{align*}
Note that the contact generator $\mathbf{c}$ remains fixed during this basis change. Moreover, Hom shows that this is a filtered change of basis, and that the pair $\{x_j', \partial x_j'\}$ splits  as a direct summand in $C^\text{vert}$. 

We iterate this process for each generator $x_j \in B_0$.  After this change of basis, we can now discard the acyclic
summand $B_0 \cup \partial B_0$ to obtain a reduced complex, since in this case, the differential strictly lowers the Alexander filtration level.   Then the only generators remaining in Alexander grading $-g$ are cycles and generate $H_*(\mathcal{F}_{-g})$. However, by \cite[Theorem 2.3]{BVV}, $H_*(\mathcal{F}_{-g}) \cong \mathbb{F}$ and  is generated by the contact class $[\mathbf{c}]$. Thus, in the reduced complex, we have that $\mathbf{c}$ is the only generator  in Alexander grading $-g$. This basis also gives a basis for a reduced complex $C$ which is filtered  chain homotopy equivalent to $CFK^-(\mathcal{H})$. The complex $C$ is reduced since the differential either drops the Alexander filtration level, or the differential drops the algebraic filtration level. 
\end{proof}

We are now ready to prove the theorem.


\begin{proof}[Proof of Theorem~\ref{thm:main}]
    Suppose $\Upsilon'_{Y,K, \mathfrak{s}}(t)=-g(K)$ for some $t\in[0,\frac{2}{n+1})$. By \cite[Theorem 3.4]{HHT} (cf.  \cite[Theorem 7.1]{Livingston}),  if $\Upsilon_{Y,K, \mathfrak{s}}(t)$ is differentiable at $t$ and a generator $[\mathbf{x},i,j]$ realizes $\Upsilon_{Y,K, \mathfrak{s}}(t)$, then $\Upsilon'_{Y,K, \mathfrak{s}}(t) = i - j = -A(\mathbf{x}).$ Since Upsilon is a group homomorphism \cite[Proposition 4.1]{ACS}, we have that $\Upsilon_{Y,K, \mathfrak{s}}(t) = -\Upsilon_{-Y,- K, \overline{\mathfrak{s}}}(t)$.

Let $\mathcal{H} = (S, \mathbf{\beta}, \mathbf{\alpha}, z, w)$ be the Heegaard diagram for $-K \subset -Y$, as described in \cite{BVV}. 
Now we will consider the complex $\CFK(\mathcal{H})$  associated to the Heegaard diagram $\mathcal{H}$.  Using Proposition~\ref{prop:contactgenerator}, we arrive at a reduced complex $C$, which is filtered chain homotopy equivalent to $\CFK(\mathcal{H})$, with the property that if a generator $\mathbf{x} \in C$ has Alexander grading $A(\mathbf{x}) = -g$, then $\mathbf{x} \in C$ must be the contact generator $\mathbf{c} \in C$. For the remainder of the proof, we work with the reduced complex $C$.

    Suppose $b(K) \leq n$. Then there exists a homogeneous chain $\mathbf{y} \in C$ with Alexander filtration level $A(\mathbf{y}) \leq  n-g$ such that $\widehat{\partial} \mathbf{y} = \mathbf{c}$. 
    Then, $\partial^\infty \mathbf{y} = \mathbf{c} + \boldsymbol{\rho}$, where $\boldsymbol{\rho} $ is nonzero since  $\mathbf{c}$ contributes to Upsilon and therefore cannot be a boundary. Observe that $\boldsymbol{\rho} \in C\{i \leq -1, j \leq n-g\}$. Then for any  integer $m$, 
    \begin{align*}
        F_t(U^{-m}\boldsymbol{\rho}) &\leq (1-\frac{t}{2})(m-1) + \frac{t}{2}(n-g+m),
        \\
        F_t(U^{-m}\boldsymbol{\rho}) - f_t(U^{-m}\mathbf{c}) &\leq (1-\frac{t}{2})(m-1) + \frac{t}{2}(n-g+m) - (1-\frac{t}{2})(m) - \frac{t}{2}(-g+m),
        \\
        F_t(U^{-m}\boldsymbol{\rho}) - f_t(U^{-m}\mathbf{c}) &\leq  \frac{t}{2}(n+1) -1.
    \end{align*}
    Thus, for $t<\frac{2}{n+1}$, we have that  $F_t(U^{-m}\boldsymbol{\rho}) < f_t(U^{-m} \mathbf{c})$. See Figure~\ref{fig:upsilonproof}. 

        Let $\boldsymbol{\eta} \in \CFK(-Y, -K, \overline{\mathfrak{s}})$ be a homogeneous cycle with $F_t(\boldsymbol{\eta}) = \nu_t(-Y, -K, \overline{\mathfrak{s}})$ where $t\in[0,\frac{2}{n+1})$ is the point where the slope of Upsilon $\Upsilon'_{Y,K, \mathfrak{s}}(t)=-g(K)$ is minimal. Then $-\Upsilon'_{-Y,-K, \overline{\mathfrak{s}}}(t)=-g(K) = A(\mathbf{c})$. We have that the class $[\boldsymbol{\eta}]$ is nonzero in homology, its Maslov grading is $d(-Y, \overline{\mathfrak{s}})$ and we write   $\boldsymbol{\eta} = U^{-m} \mathbf{c} + \boldsymbol{\eta}', $ where $\mathbf{c} = [\mathbf{c}, m, -g+m] $ is the contact generator and where $\boldsymbol{\eta}'$ does not contain a ($U$-power of) $\mathbf{c}$ summand. The integer $m$ is the $U$-power of the contact generator $\mathbf{c}$ appearing as a summand in $\boldsymbol{\eta}$ and is uniquely determined by the Maslov grading of $\boldsymbol{\eta}$. 

    As in \cite{HHT}, we  construct a cycle $\boldsymbol{\delta} = \boldsymbol{\eta} + \partial^\infty U^{-m}\mathbf{y} = \boldsymbol{\eta}' + U^{-m}\boldsymbol{\rho} $, which is homologous to $\boldsymbol{\eta}$. Moreover, the contact generator $\mathbf{c} =[\mathbf{c},0,-g]$ is not a summand of either $\boldsymbol{\rho}$ (by filtration reasons) or $\boldsymbol{\eta}'$. Moreover, since the chain $\boldsymbol{\rho}$ is in the same Maslov grading as $\mathbf{c}$, the chain $\rho$ contains no $U$-powers of $\mathbf{c}.$ Thus  no generator summand in the chain $\boldsymbol{\delta} $ lives in Alexander grading $-g$. 
    Recall that $F_t(\boldsymbol{\eta}) = \max \{ F_t(\boldsymbol{\eta'}), f_t(U^{-m}\mathbf{c})\} $ and $F_t(\boldsymbol{\delta}) = \max \{F_t(\boldsymbol{\eta}'), F_t(U^{-m}\boldsymbol{\rho})\} $. For $t<\frac{2}{n+1}$,  we have that 
    \[F_t(\boldsymbol{\eta}) \geq F_t(\boldsymbol{\delta}) \geq \nu_t(-Y,-K, \overline{\mathfrak{s}}). \]

    \begin{figure}
        \centering 
    
\begin{tikzpicture}[x=1.15cm,y=1.05cm,font=\small,>=Stealth,
  every node/.style={text=ink},
  boundary/.style={rhocolor,line width=.8pt},
  relation/.style={ink,line width=.75pt,dash pattern=on 3pt off 2.1pt}]
\def\Xmin{-4.3}
\def\Ymin{-1.45}
\def\Nheight{2.8}
\def\IllustrativeGenus{2}
\def\IllustrativeN{3}
\ifcontactlevel
  \def\LineOffset{0}
\else
  \pgfmathsetmacro{\LineOffset}{-\Nheight*\IllustrativeGenus*(\IllustrativeN-1)/(2*\IllustrativeN)}
\fi

\fill[rhofill] (\Xmin,\Ymin) rectangle (-1,\Nheight);
\begin{scope}
  \clip (\Xmin,\Ymin) rectangle (-1,\Nheight);
  \fill[pattern=north east lines,pattern color=rhocolor!17]
    (\Xmin,\Ymin) rectangle (-1,\Nheight);
\end{scope}
\draw[boundary] (\Xmin,\Nheight)--(-1,\Nheight)--(-1,\Ymin);

\begin{scope}
  \clip (\Xmin,\Ymin) rectangle (1.80,4.07);
  \draw[levelcolor,line width=1.05pt]
    (-4, {4.5*\Nheight+\LineOffset})--(1.5,{-1.9*\Nheight+\LineOffset});
\end{scope}
  \node[text=levelcolor,align=center,font=\small] at (-1.20,5.04)
    {$ t < \dfrac{2}{n+1}$};

\draw[ink!55,line width=.6pt,-{Stealth[length=5pt]}]
  (0,-1.53)--(0,4.06) node[above=2pt,text=ink] {$j$};
\draw[ink!32,line width=.45pt,dash pattern=on 2pt off 2pt]
  (-1,\Nheight)--(0,\Nheight);
\draw[ink!35,line width=.45pt,dash pattern=on 2pt off 2pt]
  (-1,\Nheight)--(-1,3.52);
\node[font=\footnotesize] at (-.60,3.60) {$i=-1$};
\node[below=4pt] at (0,-1.53) {$i=0$};
\node[above=3pt,text=rhocolor] at (-2.78,\Nheight) {$j=n-g$};

\draw[ink!60,line width=.6pt,-{Stealth[length=5pt]}]
  (.70,3.66)--(1.67,3.66) node[right=2pt,text=ink] {$i$};

\draw[relation,-{Stealth[length=5pt,width=4pt]}]
  (0,2.61)--(0,.19);
\node[right=5pt,font=\small] at (0,1.35) {$\partial^{\mathrm{vert}}$};
\draw[relation,-{Stealth[length=4.5pt,width=3.5pt]}]
  (-.10,2.72) -- (-1.55,2.00);  

\draw[fill=white,draw=ink,line width=.9pt] (0,\Nheight) circle[radius=2.6pt];
\fill[ink] (0,0) circle[radius=2.3pt];
\node[right=7pt,align=left] at (0,3.02) {$\mathbf{y}$};
\node[right=7pt,align=left,font=\small] at (0,2.71) {$(0,n-g)$};
\node[right=7pt,align=left] at (0,0.02) {$\mathbf{c} $};
\node[right=7pt,align=left,font=\small] at (0,-.31) {$(0,-g)$};
\node[right=7pt,align=left] at (-1.99, 1.80) {$\boldsymbol{\rho}$};

\node[align=center,text=rhocolor,fill=rhofill,inner sep=4pt,font=\footnotesize]
  at (-2.91,1)
  {$\operatorname{supp}(\boldsymbol{\rho})\subseteq$\\[4pt]
   $\{i\le-1,\ j\le n-g\}$};
\draw[decorate,decoration={brace,amplitude=4pt},ink!65,line width=.5pt]
  (1.70,\Nheight)--(1.70,0)
  node[midway,right=7pt] {$n$};
\end{tikzpicture}
    \caption{ The green line shows a line in the $(i,j)$-plane of slope corresponding to $t< \frac{2}{n+1}.$  If $b(K) = n$, then $\partial \mathbf{y}  = \boldsymbol{ \rho} + \mathbf{c}$, where $\boldsymbol{\rho}$ lies in the region $ \{ i \leq -1, j \leq n-g\}$. Then $\boldsymbol{\rho}$ belongs to the lower-half plane bounded by the green line, whereas the contact generator $\mathbf{c}$ does not. }
    \label{fig:upsilonproof}
    \end{figure}
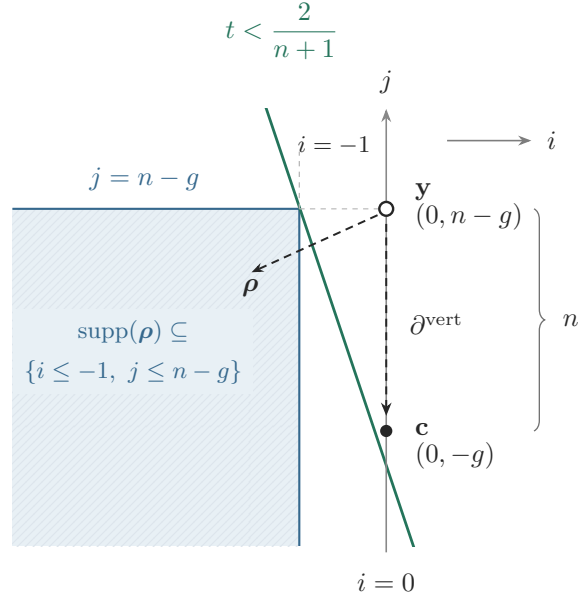
    
    The cycle $\boldsymbol{\delta}$ does not contain a $U$-power of the contact generator $\mathbf{c}$ as a summand.
    Observe that we can construct such a cycle $\boldsymbol{\delta} $ for any choice of cycle $\boldsymbol{\eta}$. However, then $\boldsymbol{\delta}$  also realizes $\nu_t(-Y,-K, \overline{\mathfrak{s}})$ since $\boldsymbol{\eta} $ does by assumption, but no summand of  $\boldsymbol{\delta} $ lies in Alexander grading $-g$ by  Proposition~\ref{prop:contactgenerator}. This contradicts that the slope of a generator realizing Upsilon is equal to negative the Alexander grading of that generator. 
    
    Thus, the slope $\Upsilon'_{-Y,- K, \overline{\mathfrak{s}}}(t) $  for $t<\frac{2}{n+1}$ cannot be equal to $g$ if $b(K) \leq n$. Finally, we conclude the slope $\Upsilon'_{Y, K, {\mathfrak{s}}}(t) $  for $t<\frac{2}{n+1}$ cannot be equal to $-g$ if $b(K) \leq n$. 
\end{proof}

For $n =1, $ we recover the right-veering obstruction of He-Hubbard-Truong. Moreover, if $K\subset S^3$ is a fibered knot which induces the unique tight contact structure on $S^3$, then $\tau(K)=g(K)$ by Hedden \cite{Hedden}. Since the right-hand slope of Upsilon at $t=0$ is $-\tau(K)$, Theorem~\ref{thm:main} implies that $b(K) = \infty$, which recovers Hedden's result that the associated contact invariant does not vanish for tight fibered knots in $S^3$.

\section{ Nontrivial terms in knot Floer homology of fibered knots}

Baldwin-Vela-Vick show that the knot Floer homology of a fibered
knot is nontrivial in its next-to-top Alexander grading. Their proof is Heegaard-diagrammatic and reduces to showing that in the case $b(K) = 1$, in which case the  term in Alexander grading $g(K) - 1$ is nonzero. We prove an analogous non-vanishing statement: the term in Alexander grading $g(K) - b(K)
$ is nonzero. 
\begin{proof}[Proof of Proposition~\ref{prop:nonzero}]
    If $b(K) = n$, then 
    \[n -g = \min \{ k \mid [\mathbf{c}] = 0 \text{ in } H_*(\mathcal{F}_k)\}.  \]
Thus, there is a chain $\mathbf{d}  \in \mathcal{F}_{n-g}$ with Alexander filtration level $n-g$  such that $\partial \mathbf{d} = \mathbf{c}$, where $\mathbf{c}$ is the contact generator. We will show that $\mathbf{d}$ induces a nonzero class in $H_*(\mathcal{F}_{n-g}/\mathcal{F}_{n-g-1})$. Since $\mathbf{c} \in  \mathcal{F}_{-g} \subset \mathcal{F}_{n-g-1}$, we have $\partial \mathbf{d} =0 $ in $\mathcal{F}_{n-g}/\mathcal{F}_{n-g-1}$. If $\mathbf{d}$ were a boundary in $\mathcal{F}_{n-g}/\mathcal{F}_{n-g-1}$, then there would exist a chain $\mathbf{e} \in \mathcal{F}_{n-g}$ such that $\partial \mathbf{e} = \mathbf{d} + \mathbf{f},$ where   $\mathbf{f} \in \mathcal{F}_{ n-g-1}$. Since $\partial^2=0$, we have $0=\partial \mathbf{d} + \partial \mathbf{f} = \mathbf{c} + \partial \mathbf{f}$. But this would imply that $[\mathbf{c}]=0$ in $H_*(\mathcal{F}_{n-g-1})$ which contradicts that $b(K) = n$. Thus, $\mathbf{d}$ induces a  nonzero class in $H_*(\mathcal{F}_{n-g}/\mathcal{F}_{n-g-1})= 
\HFKhat(-Y, -K, [-\Sigma], n-g).$  By the symmetry properties of knot Floer homology under orientation reversals and Alexander grading (cf. \cite[Remark 2.5]{BVV}), the statement follows. 
\end{proof}

 Binns-Dey \cite{BinnsDey} and Ni \cite[Theorem A.1]{Ni} extend ideas of Baldwin-Vela-Vick \cite{BVV} to prove knot Floer rank inequalities for fibered knots and links in the second-to-top Alexander grading $1-g$.   Cheng-Hedden \cite[Problem 4.6]{ChengHedden} ask how the third-to-top term of knot Floer homology is related to the monodromy of a fibered knot. We prove a nonvanishing statement for the third-to-top term by extending ideas of Ni.  
\begin{theorem}\label{thm:thirdtotop}
Let $Y$ be a closed, oriented 3-manifold. 
    If $K$ is a fibered knot in $Y$ with fiber surface $\Sigma$ whose monodromy is not right-veering and not left-veering, then 
    \[\text{rank}(\HFKhat(Y, K, [\Sigma],  2-g)) \neq 0.\]
    Moreover, there exist nonzero classes 
    \begin{align*}
    \mathbf{x} \in \HFKhat(Y, K, [\Sigma], 2-g)
    \\
    \mathbf{y}, \mathbf{a} \in \HFKhat(Y, K, [\Sigma], 1-g)
    \\
    \mathbf{c} \in \HFKhat(Y, K, [\Sigma], -g)
    \end{align*}
    such that 
    \begin{align*}
        {(d_{w})}_*(\mathbf{y}) = \mathbf{x}
        \\
        {(d_{z})}_*(\mathbf{y}) = \mathbf{c}
        \\
        {(d_{z})}_*(\mathbf{x}) = \mathbf{a}
        \\
        {(d_{w})}_*(\mathbf{c}) = \mathbf{a}
    \end{align*}
\end{theorem}
\begin{proof} Observe that $g>0$. After an isomorphism of the knot Floer complex, we can assume that the differential on the knot Floer complex strictly decreases at least one filtration level \cite[Lemma 4.5]{Rasmussen-thesis}.
Suppose the monodromy $\phi$ associated to the open book decomposition for $K$ is not left-veering. Then by work of Ni \cite[Theorem A.2]{Ni} (cf. \cite[Proof of Theorem 1.1]{BVV}) the map \[{(d_{z})}_*: \HFKhat(Y, K, [\Sigma],  1-g) \to \HFKhat(Y, K, [\Sigma],  -g)\] is nonzero. Let $\mathbf{y} \in \HFKhat(Y, K, [\Sigma],  1-g)$ be a nonzero class such that ${(d_{z})}_*(\mathbf{y}) = \mathbf{c}$. Since $\phi$ is not right-veering, we have that the map 
\[{(d_{w})}_*: \HFKhat(Y, K, [\Sigma],  -g) \to \HFKhat(Y, K, [\Sigma],  1-g)\]
is also nonzero \cite{Ni}. Since $\mathbf{c}$ is the generator of $\HFKhat(Y, K, [\Sigma],  -g)$, this means that ${(d_{w})}_*(\mathbf{c}) \neq 0.$ Let $\mathbf{a}$ be the nonzero image of $\mathbf{c}$ under the map  ${(d_{w})}_*$. Note that $\mathbf{a}$ sits in the $(-1, -g)$ bifiltration level. Since $d^2 = 0$, and the term in the $(-1, -g)$ bifiltration level of $d^2 \mathbf{y} $ corresponds to the sum ${(d_{z})}_* \circ {(d_{w})}_* \mathbf{y} + {(d_{w})}_* \circ {(d_{z})}_* \mathbf{y}$, and must vanish. Therefore, since $  {(d_{w})}_* \circ {(d_{z})}_* (\mathbf{y}) = \mathbf{a}, $ we must have that $  {(d_{z})}_* \circ {(d_{w})}_* (\mathbf{y}) = \mathbf{a}. $ In other words, there must exist a nonzero class $\mathbf{x} \in \HFKhat(Y, K, [\Sigma], 2-g)$ such that $  {(d_{w})}_* (\mathbf{y}) = \mathbf{x} $ and $  {(d_{z})}_* (\mathbf{x}) = \mathbf{a}. $ 
\end{proof}



\section{Examples}
In this section, we will give families of cable knots whose $b$-invariants are finite and grow arbitrarily large. For the  family of cables in Corollary~\ref{example:upsiloncables}, we partially compute the Upsilon functions, which attain the extremal Upsilon slope (of negative genus).  Our main theorem then implies the $b$-invariant for this family of knots grows arbitrarily large. Our computation relies on the Chen cabling inequality for Upsilon \cite{Chen}, the Hom formula for the tau invariant of cable knots \cite{Hom-cables}, and properties of the Upsilon invariant. 
\begin{corollary} \label{example:upsiloncables}
Let $p \geq 2$ and $n \geq 0$ be  integers.
For the cable knots $(T_{2, -3})_{p, pn+1}$,
the slope of $\Upsilon_{(T_{2, -3})_{p, pn+1}}(t)$ equals $-g((T_{2, -3})_{p, pn+1})$ for some $t \in (0, \frac{2}{p})$. 
 Therefore 
\[ p-1 <  b
((T_{2, -3})_{p, pn+1}) < \infty . \]
\end{corollary}
\begin{proof}
We can compute the $\tau$ invariant using Hom's formula \cite{Hom-cables} and obtain
\[\tau((T_{2, -3})_{p, pn+1}) = \frac{np(p-1)}{2} -1. \]
We can also compute the genus $ g((T_{2, -3})_{p, pn+1}) $ using Schubert's genus formula \cite{Schubert} which gives 
\[g((T_{2, -3})_{p, pn+1}) = \frac{np(p-1)}{2} + p.\]
Since the $\tau $ invariant is not equal to the genus, we conclude that the contact invariant of the associated contact structure is zero \cite{Hedden}, and thus $b((T_{2, -3})_{p, pn+1})$ is finite. 

The Upsilon invariant has right-hand slope $-\tau = -(\frac{np(p-1)}{2} -1)$ at $t=0$. The slope of Upsilon is equal to an Alexander grading of a generator of Heegaard Floer homology; for a knot in $S^3$, this must be an integer. 

The Upsilon function is a  real-valued, piecewise linear function of $t$ defined on $t \in [0,2]$, and we know so far that 
  \begin{align*}
      \Upsilon_{(T_{2, -3})_{p, pn+1}}(t) = \begin{cases}
      -(\frac{np(p-1)}{2} -1) t  & \text{for } 0 \leq t \leq t_0,
      \\
      d (t - t_0) -(\frac{np(p-1)}{2} -1) t  & \text{for } t_0 \leq t \leq t_1.
  \end{cases}
  \end{align*}
for some $t_0 > 0$, $t_1 > t_0$ and  some integer $d$. Since each slope $s$ of $\Upsilon_K(t)$ is equal to some Alexander grading s for
which $\widehat{HFK}(K, s) \neq 0$, a slope $s$ must be an integer between $-g(K)$ and $g(K)$.

For simplicity, we will define a function $h: [0,2] \to \mathbb{R}$ which shifts the Upsilon function so that the initial right-hand slope of $h$ is zero: 
\[h(t) = \Upsilon_{(T_{2, -3})_{p, pn+1}}(t) + (\frac{np(p-1)}{2} -1) t.\]
Then the above is equivalent to
  \begin{align*}
      h(t) = \begin{cases}
      0  & \text{for } 0 \leq t \leq t_0,
      \\
      d (t - t_0)   & \text{for } t_0 \leq t \leq t_1.
  \end{cases}
  \end{align*}
  Each slope of $h(t)$ is the sum of the corresponding slope $\Upsilon'_{(T_{2, -3})_{p, pn+1}}(t)$ and $\frac{np(p-1)}{2} -1$. Thus, the possible values of $d$ are integers between $-p-1$ and $(np+1)(p-1)$. 
By Proposition 1.7 of \cite{OSS}, the value $t \Delta \Upsilon'(t) = t \Delta h'(t) $ is an even integer  for any $t \in [0,2]$. Thus, at the point $t_0$, we have that $t_0 d  \in 2\mathbb{Z}$. 

Note that by an inequality of Chen for Upsilon of cable knots \cite[Theorem 1.3]{Chen}, for $\frac{1}{p} \leq t \leq \frac{2}{p}$, we have
\[ 2-2pt \leq h(t) \leq 2-(p+1)t. \]
In particular, $h(\frac{2}{p+1}) \leq 0$  and for $ \frac{2}{p+1} < t < \frac{2}{p}$ we have that $h(t) < 0$. Thus $0<t_0 \leq \frac{2}{p+1}$.

We have two cases: $d>0$ or $d<0$. If $d<0$, then the possible values of $d$ are \[-p-1, -p, \dots, -1.\] For $d<0$, we have $t_0d  \leq -2$ since $t_0 d \in 2\mathbb{Z}$; we conclude $t_0 \geq -\frac{2}{d}$. Thus, \[t_0 \geq \min\{\frac{2}{p+1}, \dots, \frac{2}{1}\} = \frac{2}{p+1},\] which shows that $t_0 = \frac{2}{p+1}$ and $d=-(p+1)$ as desired. 

In the second case, $d>0$. However, since $h(\frac{2}{p+1}) \leq 0$, the function $h(t)$ must return to zero from above at some first time $u \leq \frac{2}{p+1}$. Immediately before $u$, the function $h$ is a linear function of negative slope $d_k \in \mathbb{Z}_{<0}$. Thus we have that  $h$ is piecewise linear function with slopes $d, d_1, \dots, d_k$ and $y$-intercepts $b_0, b_1, \dots, b_k$ for $0 \leq t \leq u$. Note that $b_0=-dt_0$. 
  \begin{align*}
      h(t) = \begin{cases}
      0  & \text{for } 0 \leq t \leq t_0,
      \\
      dt + b_0   & \text{for } t_0 \leq t \leq t_1,
      \\
      d_1 t + b_1   & \text{for } t_1 \leq t \leq t_2,
      \\
      \vdots
      \\
      d_kt + b_{k}  & \text{for } t_k \leq t \leq u.
  \end{cases}
  \end{align*}
 We claim that each $y$-intercept $b_i$ is an even integer. We already have observed that $dt_0 \in 2\mathbb{Z}$. By induction, assume that $b_{i-1} \in 2\mathbb{Z}$. Then at a point $t_i$ of changing slope,
 \[d_it_i+b_i = d_{i-1} t_i + b_{i-1}.\]
 Since $t_i (d_i - d_{i-1}) \in 2\mathbb{Z}$, we conclude that $b_i \in 2\mathbb{Z}$ as well. We have $u>0$, $d_k < 0$, $b_k \in 2\mathbb{Z}_{>0} $, and
\[0=h(u) = d_k u + b_k,\]
which gives that $u = -\frac{b_k}{d_k} > 0$. Again, the possibilities for negative $d_k$ are 
 \[-p-1, -p, \dots, -1\] 
and we have
 \[u \geq \min \{ \frac{2}{p+1}, \dots, \frac{2}{1} \} = \frac{2}{p+1}. \] 
Thus, $ u = \frac{2}{p+1}$ and $d_k = -p-1$. 

Note we've shown that $\Upsilon_{(T_{2, -3})_{p, pn+1}}(t)$ attains slope $-g((T_{2, -3})_{p, pn+1})$ for some $t \in (0, \frac{2}{p})$. 
 It follows from Theorem~\ref{thm:main} that  $b((T_{2, -3})_{p, pn+1}) > p-1.$
\end{proof}


\bibliographystyle{amsalpha}
\bibliography{upsilonandb}

\end{document}